\documentclass{amsart}
\usepackage{amssymb,latexsym,amsmath,amsthm,enumitem,mathrsfs,geometry}
\usepackage{fullpage}
\usepackage{fancyhdr}
\usepackage{color}
\usepackage{stmaryrd}
\usepackage{mathrsfs}
\usepackage{hyperref}
\usepackage{lineno}
\usepackage{diagbox}
\usepackage{array}
\usepackage{tikz}
\usetikzlibrary{arrows}
\usetikzlibrary{positioning}
\usepackage{float}
\newtheorem{thm}{Theorem}[section]
\newtheorem*{thm*}{Theorem}

\newtheorem{lemma}[thm]{Lemma}

\newcommand{\beq}{\begin{equation}}
\newcommand{\eeq}{\end{equation}}

\newtheorem{theorem}{Theorem}%[section]
\newtheorem*{WMC}{Weak Mertens Conjecture}%[section]
\newtheorem*{GHC}{Gonek-Hejhal Conjecture}%[section]

\newtheorem*{proposition}{Proposition}

\definecolor{pink}{rgb}{1,.2,.6}
\definecolor{orange}{rgb}{0.7,0.3,0}
\definecolor{blue}{rgb}{.2,.6,.75}
\definecolor{green}{rgb}{.4,.7,.4}
\definecolor{purple}{RGB}{127,0,255}
\begin{document}

\numberwithin{equation}{section}

\title{A singular series average and the zeros of the Riemann zeta-function}

\author[Goldston]{D. A. Goldston}
\address{Department of Mathematics and Statistics, San Jose State University}
\email{daniel.goldston@sjsu.edu}

\author[Suriajaya]{Ade Irma Suriajaya}
\address{Faculty of Mathematics, Kyushu University}
\email{adeirmasuriajaya@math.kyushu-u.ac.jp}
\keywords{prime numbers, Riemann zeta-function, singular series}
\subjclass[2010]{11N05, 11M26}
\thanks{$^{*}$ The second author was supported by JSPS KAKENHI Grant Number 18K13400.}

\date{\today}

\begin{abstract}
We show that the Riesz mean of the singular series in the Goldbach and the Hardy-Littlewood prime-pair conjectures has an asymptotic formula with an error term that can be expressed as an explicit formula that depends on the zeros of the Riemann zeta-function. Unconditionally this error term can be shown to oscillate, while conditionally it can be shown to oscillate between sharp bounds. 
\end{abstract}

%%%%%%%%%%%%%%%%%%%%%%%%%%%%%%%%
\maketitle
%%%%%%%%%%%%%%%%%%%%%%%%%%%%%%%%

%%%%%%%%%%%%%%%%%%%%%%%%%%%%%%%%
\section{Introduction } 

This paper is concerned with averages of the arithmetic function
\begin{equation}\label{SingProd}
\mathfrak{S}(k) = \begin{cases}
{\displaystyle 2 C_{2} \prod_{\substack{ p \mid k \\ p > 2}} \!\left(\frac{p - 1}{p - 2}\right)} & \mbox{if $k$ is even, $k\neq 0$}, \\
 0 & \mbox{if $k$ is odd}\end{cases}
\end{equation}
where
\begin{equation}\label{eq1.3}
C_{2}
 = \prod_{p > 2}\! \left( 1 - \frac{1}{(p - 1)^{2}}\right)
 = 0.66016\ldots.
\end{equation}
The function $\mathfrak{S}(k)$ occurs as the arithmetic factor in conjectured asymptotic formulas for the Goldbach Conjecture and also the Hardy-Littlewood Prime Pair Conjecture and is referred to as the \lq \lq singular series" for these problems, see \cite{HardyLittlewood1919} and \cite{HardyLittlewood1922}. Averages of $\mathfrak{S}(k)$ arise frequently in a variety of problems in analytic number theory. The Ces{\`a}ro mean of $\mathfrak{S}(k)$
\begin{equation} \label{Cesaro} S_1(x) := \sum_{k\le x}(x-k)\mathfrak{S}(k) \end{equation}
appears naturally in evaluating the variance for the number of primes in a short interval, and originated in Hardy and Littlewood's unpublished paper Partitio Numerorum VII, see \cite[Lemma 8] {Rankin1940}. This variance for the number of primes in intervals is equivalent to a pair correlation type conjecture for the zeros of the Riemanm zeta-function \cite{GoldMont}, and this connection partly motivated work on obtaining precise asymptotic formulas for $S_1(x)$, see \cite{MontgomerySound1999}. The best results currently known are due to Vaughan \cite{Vaughan2001}, who proved that 
\begin{equation}\label{SingCAve} S_1(x) =\frac12 x ^2 - \frac12 x \log x + \frac12( 1-\gamma -\log 2\pi) x + E(x) \end{equation}
with
\[ E(x) \ll x^{\frac12} \exp\left( -c \frac{( \log 2x)^{\frac35}}{(\log\log 3x)^{\frac15}} \right), \]
for some positive constant $c$, and assuming the Riemann Hypothesis that 
\[ E(x) \ll x^{\frac{5}{12}+\varepsilon}.\] 
In a recent paper \cite{GS2020} we proved that
\begin{equation} E(x) = \Omega_{\pm }(x^{\frac14})\end{equation}
so that $E(x)$ cannot be too small and oscillates. 
Vaughan's method shows that the size of the error term $E(x)$ depends on a zero-free region of the Riemann zeta-function, and that there is no main term and therefore no barrier at $x^{\frac12}$. Our goal in this paper is to make this dependence between the complex zeros $\rho = \beta +i\gamma$ of the Riemann zeta-function and the error in the singular series average more explicit. The problem that is encountered in doing this for $S_1(x)$ is that the Cesaro mean is not smooth enough for absolute convergence in the contour integrals one encounters. Inspired by recent work on a related problem \cite{GZHN2017}, \cite{GeLiu2017} concerning the square of the singular series, we can sidestep this problem by using smoother weights. We found while writing this paper that this idea has already been applied to other arithmetic functions, and some results of the same type as ours have been obtained by similar methods for these functions. For $\frac{1}{\phi(k)}$ and $\frac{k}{\phi(k)}$, see \cite{GV1996}, \cite{SanSingh2013}, \cite{SanSingh2014}, and \cite{InoueKiuchi2017}, for $\frac{k}{\psi(k)}$, where $\psi$ is the Dedekind totient function, see \cite{InabaInoue2017} and \cite{InoueKiuchi2017}, and for the M\"obius function see \cite{Inoue2019} and \cite{SahaSan2019}.

We define, for $m\ge 1$,
\begin{equation} S_m(x) := \sum_{k\le x}\left(x-k\right)^m\mathfrak{S}(k). \end{equation}
This is the $m$-th Riesz mean of the singular series, although we chose to not normalize it; when $m=1$ this is the Ces{\`a}ro mean. We generalize \eqref{SingCAve} by defining $E_m(x)$, for $m\ge 1$ by 
\begin{equation} \label{formula1}
S_m(x) =\frac{1}{m+1}x^{m+1} -\frac{1}{2}x^m\left( \log x -H_m +\gamma +\log{2\pi}\right) + E_m(x), 
\end{equation}
where $H_m= 1+\frac12+\frac13+\cdots + \frac1m$ is the $m$-th harmonic number,
and we will write $E_1(x)= E(x)$ in agreement with \eqref{SingCAve}. 

To describe our main result, we note that if $\rho= \beta +i\gamma$ is a complex zero of $\zeta(s)$ then $\overline{\rho} = \beta-i\gamma$ is also a zero, so that the zeros above and below the real axis are reflections of each other. The Riemann Hypothesis is that all these complex zeros lie on the vertical line with $\beta =\frac12$, so that $\rho = \frac12 +i\gamma$. We now define
\begin{equation} \label{a_rho} a(\rho) = a(\rho,m) := \frac{ 2C_2m!\zeta(\frac{\rho}{2}-1)\zeta(\frac{\rho}{2})\mathcal{G}(\frac{\rho}{2}-1)}{(2^{\frac{\rho}{2}}+1)\zeta'(\rho) (\frac{\rho}{2}-1)(\frac{\rho}{2})(\frac{\rho}{2}+1)\cdots (\frac{\rho}{2}+m-1)},
\end{equation}
if $\rho$ is a simple zero of the Riemann zeta-function. Here
\begin{equation} \label{Gcal} \mathcal{G}(s) = \prod_{p>2}\left(1+\frac{2}{(p-2)(p^{s+1}+1)} \right). \end{equation}
If $\rho$ is a multiple zero of multiplicity $m_\rho =\ell$, then $a(\rho)$ is a polynomial of degree $\ell-1$ in powers of $\log x$ 
\begin{equation*} a(\rho) = a(\rho,m,\ell,x) = \sum_{j=0}^{\ell-1} A_j(\rho,m)(\log x)^j, \end{equation*}
defined in the next section.
When $m\ge 2$ we prove the following unconditional result.
\begin{theorem} \label{thm1}
For fixed $m\ge 2$, and any $\epsilon >0$ and $x$ sufficiently large, we have that there exists a number $U$, $x^5\le U\le 2x^5$ for which
\begin{equation} \label{Em(x)} E_m(x) = x^{m-1} \sum_{|\gamma|\le U} a(\rho) x^{\frac{\rho}{2}} + O(x^{m-1+\epsilon}). 
\end{equation}
\end{theorem} 
The sum over zeros in Theorem 1 is difficult to estimate because of our lack of knowledge of how often $\zeta'(\rho)$ can be small, and even the Riemann Hypothesis does not directly help estimate this type of sum. However, the method used to prove Theorem 1 comes within an extra factor of $x^\epsilon$ of proving the correct upper bound for $E_m(x)$.

\begin{theorem} \label{thm2} Assume the Riemann Hypothesis. For fixed $m\ge 2$, and any $\epsilon >0$ and $x$ sufficiently large, we have
\begin{equation} \label{formula2}
S_m(x) =\frac{1}{m+1}x^{m+1} -\frac{1}{2}x^m\left( \log x -H_m +\gamma +\log{2\pi}\right) + O(x^{m-\frac34+\epsilon}),
\end{equation}
and therefore
\begin{equation} \label{formula22}
E_m(x) \ll x^{m-\frac34+\epsilon}.
\end{equation}
\end{theorem}

We can conjecturally deal with the sum in Theorem 1 using conjectures related to the Mertens Hypothesis. Let 
\begin{equation} \label{M(x)} M(x) = \sum_{n\le x} \mu(n). \end{equation} 

The Mertens Hypothesis is the inequality $|M(n)| < \sqrt{n}$ for integers $n> 1$, but this was disproved in 1985 by Odlyzko and te Riele \cite{OdlyzkoteRiele}. The conjecture $M(x) \ll \sqrt{x}$ is probably false as well because Ingham \cite{Ingham1942} proved this is false assuming the reasonable conjecture that the imaginary parts of the zeros of the Riemann zeta-function are linearly independent over the integers. On the other hand, it is generally believed that the following conjecture is true.
\begin{WMC} We have
\begin{equation} \int_1^T\left( \frac{M(x)}{x}\right)^2\, dx \ll \log T. \end{equation}
\end{WMC}

It is not hard to prove that the Weak Mertens Conjecture implies the Riemann Hypothesis and that all the zeros are simple zeros, and that
\begin{equation}\label{WMCsum} \sum_{\rho} \frac{1}{|\rho\zeta'(\rho)|^2} \ll 1,\end{equation}
see \cite[14.29]{Titchmarsh}.
Ng \cite{Ng2004} found evidence that a stronger form of the Weak Mertens Conjecture is probably true. For this he made use of the following conjecture of Gonek \cite{Gonek1989} and also Hejhal \cite{Hejhal1989}.
\begin{GHC} We have
\begin{equation} \sum_{0<\gamma \le T} \frac{1}{|\zeta'(\rho)|^2} \ll T.\end{equation} 
\end{GHC}
Ng proved that assuming the Riemann Hypothesis, all zeros are simple, and the Gonek-Hejhal Conjecture, which implicitly implies all zeros are simple,
then the Weak Mertens Conjecture can be strengthened to the asymptotic formula
\begin{equation} \int_1^T\left( \frac{M(x)}{x}\right)^2\, dx \sim \sum_{\rho} \frac{2}{|\rho\zeta'(\rho)|^2}\log T . \end{equation}
For more recent results on this topic, see \cite{Inoue2019} and \cite{SahaSan2019}. 

\begin{theorem} \label{thm3} Assume the Riemann Hypothesis and the simplicity of zeros, so that $\rho=\frac12 +i\gamma$ is simple. If in addition the Gonek-Hejhal Conjecture is true, then for $m\ge 2$ and any $\epsilon >0$ and $x$ sufficiently large, we have
\begin{equation} \label{EmExplicit} E_m(x) = x^{m-\frac34} \sum_{\gamma} \frac{ 2C_2m!\zeta(\frac{\rho}{2}-1)\zeta(\frac{\rho}{2})\mathcal{G}(\frac{\rho}{2}-1)}{(2^{\frac{\rho}{2}}+1)\zeta'(\rho) (\frac{\rho}{2}-1)(\frac{\rho}{2})(\frac{\rho}{2}+1)\cdots (\frac{\rho}{2}+m-1)}x^{i\frac{\gamma}{2}} + O(x^{m-1+\epsilon}), \end{equation}
where the sum over zeros is absolutely convergent. Therefore 
\begin{equation} \label{EmBigO} E_m(x) \ll x^{m-\frac34}, \end{equation}
and more precisely,
letting $c_m := \sum_{\gamma}|a(\rho)|$, where $a(\rho)$ is as defined in \eqref{a_rho}, we have \begin{equation}\label{Embound} |E_m(x)| \le (1+o(1))c_m x^{m-\frac34}. \end{equation}
For $m\ge 3$ the same results hold when we assume the Weak Mertens Conjecture.
\end{theorem}

Unconditionally we prove that the bounds in Theorem \ref{thm3} are sharp up to constants.

\begin{theorem} \label{thm4} For $m\ge 1$, we have 
\begin{equation} E_m(x) = \Omega_{\pm}(x^{m-\frac34}). \end{equation}
\end{theorem}

Conditionally we can strengthen Theorem \ref{thm4} and show that the bounds in \eqref{Embound} are sharp. 

\begin{theorem} \label{thm5} If the Riemann zeta-function has a complex zero off the half-line or a zero with multiplicity 2 or higher, then for $m\ge 1$
\begin{equation} \label{thm5eq1} \limsup_{x\to\infty} \frac{E_m(x)}{x^{m-\frac34}} = \infty \qquad \text{and} \qquad \liminf_{x\to\infty} \frac{E_m(x)}{x^{m-\frac34}} = -\infty .\end{equation}
For the alternative situation, we assume the Riemann Hypothesis is true and all the zeros are simple. Now let $\mathcal{S}$ be the set of distinct positive imaginary parts of the zeros of the Riemann zeta-function, and assume the hypothesis that every finite subset of $\mathcal{S}$ is linearly independent over the integers. If $\sum_\gamma|a(\rho)|$ diverges then \eqref{thm5eq1} holds. However if $m\ge 2$ and $c_m =\sum_\gamma|a(\rho)|$ converges, then we have
\begin{equation}\label{thm5eq2} \limsup_{x\to\infty} \frac{E_m(x)}{x^{m-\frac34}} \ge c_m \qquad \text{and} \qquad \liminf_{x\to\infty} \frac{E_m(x)}{x^{m-\frac34}} \le -c_m. \end{equation}
When $m=1$, $\sum_\gamma|a(\rho)|$ diverges, and therefore writing $E_1(x)=E(x)$, we have, assuming the linear independence hypothesis,
\begin{equation} \label{thm5eq3} \limsup_{x\to\infty} \frac{E(x)}{x^{\frac14}} = \infty \qquad \text{and} \qquad \liminf_{x\to\infty} \frac{E(x)}{x^{\frac14}} = -\infty . \end{equation}
\end{theorem}
\noindent By Theorem \ref{thm3} the upper bound in \eqref{Embound} together with \eqref{thm5eq2} implies that for $m\ge 2$
\begin{equation}\label{thm3+5} \limsup_{x\to\infty} \frac{E_m(x)}{x^{m-\frac34}} = c_m \qquad \text{and} \qquad \liminf_{x\to\infty} \frac{E_m(x)}{x^{m-\frac34}} = -c_m. \end{equation}
This is under the assumption of the Riemann Hypothesis, simple zeros, the Gonek-Hejhal Conjecture, and the linear independence hypothesis. This result can be proved more easily by directly applying Kronecker's Theorem in the explicit formula \eqref{EmExplicit} of Theorem \ref{thm3}.

Theorem \ref{thm4} was proved in \cite{GS2020} in the case $m=1$ and the same proof with minor modifications works for all $m\ge 1 $. Here we will prove Theorems \ref{thm4} and \ref{thm5} by making use of a theorem of Ingham \cite{Ingham1942}. In proving Theorem \ref{thm5} this allows us to avoid the Gonek-Hejhal Conjecture or the Weak Mertens Conjecture by providing a replacement for \eqref{EmExplicit}.

The proof of \eqref{thm5eq2} requires that the imaginary parts of the zeros are linearly independent over all the integers, but following \cite{Bateman...71} and later authors, we can prove that the linear independence needed to obtain \eqref{thm5eq3} can be limited to coefficients $-2,-1,0,1,2$, with at most one coefficient of $2$ or $-2$. The proof uses Ingham's theorem and a Riesz product that first occurred in a proof of Kronecker's theorem by Bohr and Jessen,\cite{BJ1932}, see also \cite[9.3]{Katznelson1976}.

%In \cite{G-S20b} we use the Landau oscillation theorem to show that the linear independence conjecture for imaginary parts of zeros of the Riemann zeta-function alone is sufficient when $m\ge 2$ to prove oscillations of $E_m(x)$ at least as large as in Theorem 5. When $m=1$ we obtain even larger oscillations, subject to a weaker linear independence conjecture. As a special case we unconditionally obtain another proof of Theorem 4.

%%%%%%%%%%%%%%%%%%%%%%%%%%%%%%%%
\section{The Inverse Mellin Transform of %$S_m(x)$}
\texorpdfstring{$S_m(x)$}{Sm(x)}}

Define the singular series generating function
\begin{equation} \label{F} F(s) = \sum_{k=1}^\infty \frac{\mathfrak{S}(k)}{k^s}, \end{equation}
where as usual $s=\sigma + i t$. This series converges for $\sigma >1$. The first lemma from \cite{GS2020} provides the analytic continuation of $F(s)$ to $\sigma >-1$. 

\begin{lemma} \label{F(s)lem} We have, for $\sigma >-1$, 
\begin{equation} \label{F(s)formula}
F(s) 
= \left(\frac{4C_2}{2^{s+1}+1} \right) \frac{ \zeta(s)\zeta(s+1)}{\zeta(2s+2)}\mathcal{G}(s),\end{equation}
where $\mathcal{G}(s)$ is given in \eqref{Gcal}.
\end{lemma} 

We now evaluate $S_m(x)$ as a contour integral. Using the formula in Theorem B of Ingham \cite{Ingham1932}, if $m$ is a positive integer, $c>0$, and $x>0$, then we have
\begin{equation} \label{CesaroContour} \frac{m!}{2\pi i} \int_{c-i\infty}^{c+i\infty} \frac{x^{s+m}}{s(s+1)(s+2) \cdots (s+m)} \, ds = \begin{cases}
0 & \mbox{if $0< x\le 1$}, \\
(x-1)^m & \mbox{if $ x\ge 1$. }\end{cases}
\end{equation}
Hence, for $c>1$,
\begin{equation} \label{Scontour} S_m(x) = \frac{1}{2\pi i} \int_{c-i\infty}^{c+i\infty}\mathcal{F}(s) \, ds , \end{equation}
where 
\begin{equation} \label{calF} \mathcal{F}(s) := \frac{m! F(s) x^{s+m}}{s(s+1)(s+2) \cdots (s+m)} = \frac{4C_2 m!\zeta(s)\zeta(s+1)\mathcal{G}(s)x^{s+m}}{(2^{s+1}+1)\zeta(2s+2)s(s+1)(s+2) \cdots (s+m)}. \end{equation}
We see by Lemma \ref{F(s)lem} that for $\sigma > -1$ the singularities of $\mathcal{F}(s)$ consist of a simple pole at $s=1$, a double pole at $s=0$, and poles of order $m(\rho)$ at $s= \rho/2-1$, where $\rho$ denotes any
complex zero of $\zeta(s)$. We need to compute the residues at these poles. Letting $R_{\mathcal{F}}(a) = \mbox{\text Res}(\mathcal{F}(s);s=a)$, we collect our results in the following lemma. The case when $m=1$ was already done in \cite{GS2020}.

\begin{lemma}\label{lemma2.2} The residues of the singularities of $\mathcal{F}(s)$ for $\sigma >-1$ are
\begin{equation} R_{\mathcal{F}}(1) = \frac{x^{m+1}}{m+1}, \end{equation}
 \begin{equation} R_{\mathcal{F}}(0) = -\frac12\left( \log x -H_m +\gamma +\log{2\pi}\right)x^m, \end{equation}
where $H_m = \sum_{n=1}^m \frac1n$ is the $m$-th Harmonic number,
\begin{equation} R_{\mathcal{F}}(\frac{\rho}{2}-1) =\frac{ 2C_2m!\zeta(\frac{\rho}{2}-1)\zeta(\frac{\rho}{2})\mathcal{G}(\frac{\rho}{2}-1)}{(2^{\frac{\rho}{2}}+1)\zeta'(\rho) (\frac{\rho}{2}-1)(\frac{\rho}{2})(\frac{\rho}{2}+1)\cdots (\frac{\rho}{2}+m-1)}x^{\frac{\rho}{2}+m-1},
\qquad \text{if} \ \ m_\rho=1,
%= a(\rho)x^{\frac{\rho}{2}+m-1}
\end{equation}
and in general
\begin{equation} \label{ellresidue} R_{\mathcal{F}}(\frac{\rho}{2}-1) =x^{\frac{\rho}{2}+m-1}\sum_{j=0}^{\ell-1} A_j(\rho,m)(\log x)^j, \qquad \text{if} \ \ m_\rho=\ell, \end{equation}
for constants $A_j(\rho,m)$.
\end{lemma}

\begin{proof}
For the simple pole at $s=1$ with residue 1 of $\zeta(s)$, we have 
\[ R_{\mathcal{F}}(1) = \frac{4C_2 m!\zeta(2) \mathcal{G}(1) x^{m+1}}{5\zeta(4)(m+1)!} = \frac{x^{m+1}}{m+1} \]
since
\[ \begin{split} \frac{4\zeta(2) \mathcal{G}(1)}{5\zeta(4)} &= \prod_{p>2} 
\left( 1-\frac{1}{p^2}\right)^{-1} \left( 1-\frac{1}{p^4}\right)\left(1+\frac{2}{(p-2)(p^2+1)} \right)\\ &
=\prod_{p>2} 
\left(\frac{(p-1)^2}{p(p-2)}\right) \\ &
= 1/C_2 . \end{split} \]

Next, for the double pole at $s=0$ of $\zeta(s+1)/s$, we let 
\[ U(s) := \frac{s}{\zeta(s+1)}\mathcal{F}(s)
= \frac{4C_2 m!\zeta(s)\mathcal{G}(s)x^{s+m}}{(2^{s+1}+1)\zeta(2s+2)(s+1)(s+2) \cdots (s+m)}.\]
Since $U(s)$ is analytic at $s=0$ and $\zeta(s) = \frac{1}{s-1} + \gamma + O(|s-1|)$ for $s$ near $1$, see \cite[Eq. (2.1.16)]{Titchmarsh}, we obtain the Laurent expansion around $s=0$
\[\begin{split} \mathcal{F}(s) &= \frac{\zeta(s+1)}{s} U(s)\\& = \left(\frac{1}{s^2} + \frac{\gamma}{s} + O(1)\right)\left( U(0) + sU'(0) + O(|s|^2)\right)\\&
= \frac{U(0)}{s^2} + \frac{\gamma U(0) +U'(0)}{s} +O(1), \end{split} \]
and therefore
\[ R_{\mathcal{F}}(0) = \left(\frac{U'(0)}{U(0)} +\gamma\right) U(0). \]
Here $\gamma \simeq 0.577216$ is the Euler-Mascheroni constant. We will make use below of the special values 
$$
\frac{\zeta'(0)}{\zeta(0)} = \log{2\pi}, \qquad \text{and} \qquad \zeta(0) = - \frac12, $$
see \cite[2.4]{Titchmarsh}, and also the harmonic numbers $H_m= \sum_{n=1}^m \frac1n$.
First,
$$
U(0) = \frac{4C_2 \zeta(0)\mathcal{G}(0)x^{m}}{3\zeta(2)}
= -\frac12 x^{m} 
$$
since 
\[ \begin{split} \frac{4 \mathcal{G}(0)}{3\zeta(2)} &= \prod_{p>2} 
\left( 1-\frac{1}{p^2}\right) \left(1+\frac{2}{(p-2)(p+1)} \right)\\ &
=\prod_{p>2} 
\left(\frac{(p-1)^2}{p(p-2)}\right) \\ &
= 1/C_2 . \end{split} \]
Next, logarithmically differentiating $U(s)$ and evaluating at $s=0$, we obtain
\[\begin{split}
\frac{U'(0)}{U(0)} 
&= 
 \frac{\zeta'}{\zeta}(0) + \frac{\mathcal{G}'}{\mathcal{G}}(0) + \log{x} -\frac23\log{2}-2\frac{\zeta'}{\zeta}(2) - H_m
 \\ & = \log x +\log 2\pi - H_m, \end{split} \]
since
%$$
%H_m = \sum_{n=1}^m \frac1n = \gamma + \log{m} + O\left(\frac1n\right).
%\qquad \gamma = \lim_{m\to\infty} \left( H_m - \log{m} \right).
%$$
\[ \begin{split} \frac{\mathcal{G}'}{\mathcal{G}}(0) -\frac23\log{2}-2\frac{\zeta'}{\zeta}(2)
&= \sum_{p>2} \frac{-2p\log{p}} {(p-2)(p+1)^2\left(1+\frac{2}{(p-2)(p+1)} \right)}+ \sum_{p>2} \frac{2\log{p}}{p^2-1} \\& = 0. \end{split} \]
Therefore we conclude
\[ R_{\mathcal{F}}(0) = -\frac12\left( \log x -H_m +\gamma +\log{2\pi}\right)x^m .\]

Finally, we evaluate $R_{\mathcal{F}}(\frac{\rho}{2}-1)$ where the $\rho$'s are the complex zeros of $\zeta(s)$. If $\rho$ is a simple zero of $\zeta(s)$, then $\frac{1}{\zeta(2s+2)}$ will have a corresponding simple pole at $s= \frac{\rho}{2}-1$ with residue $\frac{1}{2\zeta'(\rho)}$. Thus
\[ R_{\mathcal{F}}(\frac{\rho}{2}-1) = \frac{ 2C_2m!\zeta(\frac{\rho}{2}-1)\zeta(\frac{\rho}{2})\mathcal{G}(\frac{\rho}{2}-1)}{(2^{\frac{\rho}{2}}+1)\zeta'(\rho) (\frac{\rho}{2}-1)(\frac{\rho}{2})(\frac{\rho}{2}+1)\cdots (\frac{\rho}{2}+m-1)}x^{\frac{\rho}{2}+m-1}. \]

In the general case of a zero $\rho$ with multiplicity $m_\rho= \ell$, we expand around $s= \rho/2-1$ and have
\[ x^{s+m} = x^{\frac{\rho}{2}+m-1} x^{s-\frac{\rho}{2}+1} 
= x^{\frac{\rho}{2}+m-1}\sum_{k=0}^\infty \frac{ ((s-\frac{\rho}{2}+1)\log x)^k}{k!}, \]
while the remaining part of $\mathcal{F}(s)$ expands into a Laurent expansion 
\[ \frac{4C_2 m!\zeta(s)\zeta(s+1)\mathcal{G}(s)}{(2^{s+1}+1)\zeta(2s+2)s(s+1)(s+2) \cdots (s+m)} = \sum_{j= -\ell}^\infty \mathfrak{a}_j(\rho,m)(s-\frac{\rho}{2}+1)^j \]
which does not depend on $x$, and one obtains \eqref{ellresidue} on multiplying out and collecting the coefficients of $(s-\frac{\rho}{2}+1)^{-1}$, that is in \eqref{ellresidue},
$$
A_j(\rho,m) := \mathfrak{a}_{-j-1}(\rho,m)\frac{\log x)^j}{j!}.
$$
\end{proof}

%%%%%%%%%%%%%%%%%%%%%%%%%%%%%%%%
\section{Proof of Theorem \ref{thm1}}

In this section we prove Theorem 1. We start with the inverse Mellin transform formula \eqref{Scontour} for $S_m(x)$ and move the contour to the left. To justify this we need various bounds for $\mathcal{F}(s)$.

For the Riemann zeta-function, we have the classical bound
\begin{equation} \label{t-bound}
\zeta(s)-\frac{1}{s-1} \ll_{\epsilon} |t|^{\mu(\sigma)+\epsilon}, \qquad s=\sigma +it,
\end{equation}
which holds for any $\epsilon>0$ with
\begin{equation} \label{mu}
\mu(\sigma) \leq
\begin{cases}
0 &\text{ if } \sigma>1, \\
(1-\sigma)/2 &\text{ if } 0\leq\sigma\leq1, \\
1/2-\sigma &\text{ if } \sigma<0,
\end{cases}
\end{equation}
(see \cite[5.1]{Titchmarsh}).
Next, we need the zero-free region of $\zeta(s)$ \cite[Theorem 3.8]{Titchmarsh})
\begin{equation} \label{zero-free}
\sigma \ge 1 - \frac{A}{\log{(|t|+3)}}
\end{equation}
for some $A>0$. %\footnote{From work of Ford the value $A= \frac1{20}$ is acceptable.}.
In this region, the reciprocal of $\zeta(s)$ is analytic and satisfies the bound
\begin{equation} \label{zeta_recip}
\frac{1}{\zeta(\sigma+it)} \ll \log{(|t|+3)},
\end{equation}
see \cite[Eq. (3.11.8)]{Titchmarsh}, and for values of $A$ and the constant in \eqref{zeta_recip}, see \cite{Trudgian2015}. 

By the functional equation, we have for any fixed $B>0$ that
\begin{equation}\label{FE} |\zeta(s)| \asymp (|t|+3)^{\frac12-\sigma}|\zeta(1-s)| \end{equation} 
holds uniformly for $|\sigma|\le B$ and $|t|\ge 1$, see \cite[Corollary 10.5]{MontgomeryVaughan2007}.
We therefore conclude that $\zeta(s)$ has no zeros in the region 
\[-B\le \sigma\le \frac{A}{\log{(|t|+3)}}, \qquad |t|\ge 1, \]
and in this region 
\[\frac{1}{\zeta(\sigma+it)} \ll (|t|+3)^{\epsilon +\sigma -\frac12}. \]
Finally, we will make use of a nice result of Ramachandra and Sankaranarayanan \cite[Theorem 2]{RS1991} which allows us to avoid the need to assume the Riemann Hypothesis. For $T$ sufficiently large and $C>0$ a constant, we have
\begin{equation} \label{Ram-San} \min_{T\le t\le T+T^{1/3}} \max_{\frac12\le \sigma \le 2} |\zeta(\sigma +it)|^{-1} \le \exp(C(\log\log T)^2). \end{equation}
From this result and \eqref{FE} we conclude that for any $\epsilon>0$ and $T$ sufficiently large there exist a $T'$, $T\le T'\le T+T^{\frac13}$, such that 
\begin{equation} \label{Horizontal} |\zeta(\sigma +iT')|^{-1}\ll (T')^{-\max(\frac12-\sigma,0) + \epsilon }\ll (T')^{ \epsilon} \end{equation}
for all $-B\le \sigma \le B$.

\begin{proof}[Proof of Theorem \ref{thm1}]
We now assume $m\ge 2$ is a fixed integer, and do not include the dependence on $m$ in error estimates. Taking $c=2$ in \eqref{Scontour}, we have
\begin{equation} \label{Sc=2} S_m(x) = \frac{1}{2\pi i} \int_{2-i\infty}^{2+i\infty}\mathcal{F}(s) \, ds , \end{equation}
where 
\begin{equation} \label{calF2} \mathcal{F}(s) = \frac{4C_2 m!\zeta(s)\zeta(s+1)\mathcal{G}(s)x^{s+m}}{(2^{s+1}+1)\zeta(2s+2)s(s+1)(s+2) \cdots (s+m)}. \end{equation}
Since all the zeta-functions in $\mathcal{F}(s)$ are bounded and have no zeros when $s=2+it$, we have
\[ \mathcal{F}(2+it) \ll \frac{x^{2+m}}{(|t|+3)^{ m+1}},\]
and therefore we can truncate the integral in \eqref{Sc=2} to the range $|t|\le T$ with an error $\ll x^{2+m}/T^{ m}$. We conclude that for $m\ge 2$ and $T\ge x^{2}$
\begin{equation} \label{truncated} S_m(x) = \frac{1}{2\pi i} \int_{2-iT}^{2+iT}\mathcal{F}(s) \, ds +O(1). \end{equation}

We now consider the rectangle with corners $2\pm iT^*$ and $b\pm iT^*$, where $b=-1+\frac{a}{\log T}$ and $T^*$ is chosen as in \eqref{Horizontal} so that $T\le T^*\le 2T$, and $a$ %\footnote{Here $A$ needs to be twice smaller, e.g. $A= \frac1{40}$.} 
is chosen small enough so that all the poles of $1/\zeta(2s+2)$ at $s= \rho/2 -1$ with $|\gamma|/2 \le T^*$ have $\beta/2 > a/\log T^*$. Thus all the singularities of $\mathcal{F}(s)$ for $-1 < \sigma \le 2$ and $ |t| \le T^*$ are inside the rectangle and not on the boundary. On the horizontal contours $s= \sigma \pm iT^*$ and for $-1\le \sigma \le 2$ we have 
\[\frac{1}{\zeta(2s+2)}\ll (T^*)^\epsilon.\]
On the vertical contour $s= b +it$, $|t|\le T^*$ we have
\[\frac{1}{\zeta(2s+2)} \ll (|t|+3)^{2b +\frac32 +\epsilon}.\] 

We integrate around this contour counterclockwise and obtain by the residue theorem and Lemma 2.2 that 
\begin{equation}\label{1asym} S_m(x) = R_{\mathcal{F}}(1) + R_{\mathcal{F}}(0) + \sum_{|\gamma |\le 2 T^*} R_{\mathcal{F}}(\rho/2-1)+\frac{1}{2\pi i} \left( \int_{2-iT^*}^{b-iT^*} + \int_{b-iT^*}^{b+iT^*} +\int_{b+iT^*}^{2+iT^*} \right) \mathcal{F}(s) \, ds +O(1). \end{equation}

We now bound the contribution from the contour integrals along the three sides of the rectangle. To do this we need a bound on $\mathcal{G}(s)$ provided by the following lemma. 
\begin{lemma} \label{g_bound} For any fixed $\delta>0$ and $-1+\delta \le \sigma \le 2$, we have
\begin{equation}\label{calG bounded} \mathcal{G}(s) \asymp_\delta 1. \end{equation}
Further, with $b=-1+\frac{a}{\log T}$, we have uniformly for $b\le \sigma \le 2$ and $|t|\ll T$
\begin{equation} \label{calGupper} \mathcal{G}(s) \ll \exp\left( B \frac{\log T}{\log\log T}\right),\end{equation}
where $B$ is a constant. In particular $\mathcal{G}(s) \ll T^\epsilon$ for any $\epsilon >0$ and $T$ sufficiently large.
\end{lemma}
We prove this lemma after completing the proof of Theorem 1.

Consider first the horizontal line segments $\sigma \pm iT^*$ with $b\le \sigma \le 2$. Applying the estimates above we have 
\[\begin{split} \mathcal{F}(\sigma \pm i T^*) &\ll \frac{{T}^{ \mu(\sigma)+\mu(\sigma+1)+ 4\epsilon}}{{T}^{m+1}}x^{\sigma + m} \\&
\ll \frac{{T}^{ \mu(-1)+\mu(0)+ 4\epsilon}}{{T}^{m+1}}x^{ m+2}
\\ &
\ll {T}^{1-m +4\epsilon}x^{m+2}\\&
\ll 1 \end{split} \]
when $T \ge x^5$ since $m\ge 2$. Hence the contribution from these horizontal contours is $\ll 1$.

Next, on the vertical contour with $\sigma = b +it$, $|t|\le T^*$ we have 
\begin{equation} \label{Fvert} \mathcal{F}(b + i t) \ll (|t|+3)^{\frac12 -m}T^\epsilon x^{m+b}, \end{equation}
since for the range
$1\le |t|\le T^*$ we have
\[\begin{split} \mathcal{F}(b + i t) &\ll \frac{(|t|+3)^{ \mu(b) +\mu(b+1) + 2b+ \frac32 +3\epsilon}}{(|t|+3)^{m+1}}T^\epsilon x^{m+b} \\ &
\ll (|t|+3)^{\mu(-1) +\mu(0)-\frac32-m+3\epsilon }T^\epsilon x^{m+b} \\ &
\ll (|t|+3)^{\frac12-m+3\epsilon}T^\epsilon x^{m+b}, \end{split} \]
while for $|t|\le 1$ we have
$$
\mathcal{F}(b+it)
\ll \frac{x^{m+b}}{|b+1+it|}\left|\mathcal{G}(b+it)\right|
\ll \frac{x^{m+b}}{|b+1+it|}T^\epsilon \ll 
x^{m+b}T^{\epsilon}\log T\ll x^{m+b}T^{2\epsilon},
$$
where we applied Lemma \ref{g_bound} in the second inequality.
 
Taking $T=x^5$ and integrating \eqref{Fvert} over $|t|\le T^*$ the contribution from this vertical contour is $\ll x^{m-1+\epsilon}$.
\end{proof}

\begin{proof}[Proof of Lemma \ref{g_bound}]
For $-1+\delta \le \sigma \le 2$ we have
\[ \prod_{p>2}\left( 1- \frac{2}{p^{1+\delta}}\right)\ll |\mathcal{G}(s)|\ll \prod_{p>2}\left( 1+ \frac{2}{p^{1+\delta}}\right),\]
and \eqref{calG bounded} follows because the sums here converge for any fixed $\delta>0$.
For $ b \le \sigma \le 2$ we have
\[ |\mathcal{G}(s)| \le \prod_{p>2}\left( 1+ \frac{2}{(p-2)(p^{b +1}-1)}\right).\]
Now
\[ p^{b+1}- 1 =\int_0^{b+1} p^u\log p \, du \ge (b+1)\log p = \frac{a\log p}{\log T}, \]
and therefore
\[ |\mathcal{G}(s)| \le \prod_{p>2}\left( 1+ \frac{2\log T}{a(p-2)\log p}\right).\]
The terms in the product with $p\le Y :=\frac4a \frac{\log T}{\log\log T}$ make a contribution
\[ \ll \prod_{p\le Y} \frac{2}{a} \log T = (\frac{2}{a} \log T)^{\pi(Y)} \ll \exp(C_1 \frac{\log T}{\log\log T}) .\]
Using the inequality $\log(1+x) \le x$ for $0\le x <1$, the terms in the product with $p>Y$ make a contribution
\[\begin{split} & = \exp\left( \sum_{p>Y} \log\left( 1+ \frac{2\log T}{a(p-2)\log p}\right) \right) \\& \leq \exp\left( \sum_{p>Y} \frac{2\log T}{a(p-2)\log p} \right) \\
&\ll \exp\left(C_2 \log T \left(\int_{Y}^\infty \frac{du}{u\log^2 u} + \frac1{\log^2Y}\right)\right) \\
&\ll \exp\left(\frac{C_2 \log T}{\log Y}\right) \\
&\ll \exp\left(\frac{2C_2 \log T}{\log \log T}\right).
\end{split}\]
Multiplying these contributions prove \eqref{calGupper} with $B=C_1+2C_2$.
\end{proof}

\section{Proof of Theorem \ref{thm2}}

Starting with \eqref{truncated}, we use the same rectangle as before but now take $b=b_1$, a fixed number with $ -\frac34 <b_1< -\frac12$. By the Riemann Hypothesis $\rho/2-1 = -\frac34 +i\gamma/2$, and therefore none of the singularities of $1/\zeta(2s+2)$ are in this rectangle. Hence \eqref{1asym} becomes 
\begin{equation}\label{RH1asym} S_m(x) = R_{\mathcal{F}}(1) + R_{\mathcal{F}}(0) +\frac{1}{2\pi i} \left( \int_{2-iT^*}^{b_1-iT^*} + \int_{b_1-iT^*}^{b_1+iT^*} +\int_{b_1+iT^*}^{2+iT^*} \right) \mathcal{F}(s) \, ds +O(1). \end{equation}
As before when $T\ge x^5$ and $m\ge 2$ the integrals along the horizontal contours are $O(1)$. On the Riemann Hypothesis, we have
$$ \frac{1}{\zeta(\sigma+it)} \ll (|t|+3)^\epsilon,
\qquad \text{for every}~\sigma>1/2, $$
see \cite[Eq. (14.2.6)]{Titchmarsh} and we may replace the bound for $\mu(\sigma)$ in \eqref{mu} by
\begin{equation} \label{RHmu} \mu(\sigma) \le \max( \frac12 -\sigma, 0).\end{equation}
Therefore on the vertical contour with $\sigma = b_1 +it$, $|t|\le T^*$, \eqref{Fvert} becomes
\begin{equation} \label{RHFvert} \mathcal{F}(b_1 + i t) \ll (|t|+3)^{-2b_1- m-1}T^\epsilon x^{m+b_1}. \end{equation}
On integrating we obtain on taking $T=x^5$ that the contribution from this vertical contour is $\ll x^{m +b_1+\epsilon}$. Since $b_1$ can be taken as close to $-\frac34$ as we wish, this proves Theorem 2.

\section{Proof of Theorem \ref{thm3}}

We first prove the following lemma.
\begin{lemma}[Ng, 2004] \label{Gammasum-lemma} Assuming the Riemann Hypothesis, the simplicity of all zeros, and the Gonek-Hejhal Conjecture, then the sums
\begin{equation} T_1(b) := \sum_{\gamma} \frac{1}{|\gamma|^b|\zeta'(\rho)|}\qquad \text{and}\qquad T_2(b):= \sum_{\gamma} \frac{1}{|\gamma|^b|\zeta'(\rho)|^2} \end{equation}
both converge for any fixed number $b>1$. Assuming instead the Weak Mertens Conjecture these sums converge for any $b\ge 2$.
\end{lemma}
\noindent This lemma is partially contained in \cite[Lemma 1]{Ng2004}.

\begin{proof}[Proof of Lemma \ref{Gammasum-lemma}]
By the well-known unconditional estimate \cite[Theorem 25b]{Ingham1932}
\[U(T) := \sum_{0<\gamma \le T}\frac1{\gamma} \ll \log^2 T,\]
we have $ \sum_\gamma \frac{1}{|\gamma|^b} $
converges for $b>1$ since
\[\begin{split} \sum_{0<\gamma\le T} \frac{1}{\gamma^b} &= \int_1^T \frac{1}{t^{b-1}}\, dU(t) \\ &
= \frac{U(T)}{T^{b-1}} + (b-1)\int_1^T\frac{U(t)}{t^b}\, dt \\&
\ll 1 .\end{split}\]
Hence by Cauchy's inequality 
\[ T_1(b) \le \sqrt{T_2(b)\sum_{\gamma}\frac{1}{|\gamma|^b}}, \]
and therefore when $b>1$ the convergence of $T_1(b)$ follows from the convergence of $T_2(b)$. By \eqref{WMCsum}
the Weak Mertens Conjecture immediately shows $T_2(b)$ converges for $b\ge 2$. Denoting
\[ V(T) := \sum_{0<\gamma \le T} \frac{1}{|\zeta'(\rho)|^2},\]
then for $b>1$
\[\begin{split} T_2(b) &= \int_1^\infty \frac{1}{t^{b}}\, dV(t) \\
&= b\int_1^\infty\frac{V(t)}{t^{b+1}}\, dt \\
&\ll 1 ,\end{split}\]
since by the Gonek-Hejhal Conjecture $V(t) \ll t$.
\end{proof}

\begin{proof}[Proof of Theorem \ref{thm3}]
By \eqref{calG bounded} and the Riemann Hypothesis bounds used to obtain \eqref{RHFvert} we have
\[ |a(\rho)| \ll \frac{1}{|\gamma|^{m-\frac12 -\epsilon}|
\zeta'(\rho)|}. \]
Here, assuming the simplicity of zeros, $a(\rho)$ is as defined in \eqref{a_rho}.
Hence by Lemma \ref{Gammasum-lemma} we have that $\sum_\gamma a(\rho)$ is absolutely convergent for $m\ge 2$ on the Gonek-Hejhal Conjecture and for $m\ge 3$ on the Weak Mertens Conjecture. To prove \eqref{EmExplicit}, by \eqref{Em(x)} we know that there exists a number $U$ satisfying $x^5\le U\le 2x^5$ for which, assuming the Riemann Hypothesis,
\[ \begin{split}
E_m(x) 
&= x^{m-\frac34} \sum_{|\gamma|\le U} a(\rho) x^{i\frac{\gamma}{2}} + O(x^{m-1+\epsilon}) \\
&= x^{m-\frac34} \sum_{\gamma} a(\rho) x^{i\frac{\gamma}{2}}
- x^{m-\frac34} \sum_{|\gamma|> U} a(\rho) x^{i\frac{\gamma}{2}}
+ O(x^{m-1+\epsilon}). \end{split}
\]
By Lemma \ref{Gammasum-lemma}, 
\begin{align*}
\sum_{|\gamma|> U} a(\rho) x^{i\frac{\gamma}{2}}
\ll \sum_{|\gamma|> U} \frac1{|\gamma|^{m-\frac12-\epsilon}|\zeta'(\rho)|}
\le \frac{1}{U^{\frac15}}\sum_{|\gamma|>U} \frac1{|\gamma|^{m-\frac34}|\zeta'(\rho)|}\le\frac{1}{U^{\frac15}}\sum_{|\gamma|} \frac1{|\gamma|^{m-\frac34}|\zeta'(\rho)|}\ll \frac{1}{U^{\frac15}}
\ll \frac1{x},
\end{align*}
for $m\ge 2$ on the Gonek-Hejhal conjecture and for $m\ge 3$ on the Weak Mertens Conjecture.
Thus we obtain \eqref{EmExplicit}, and equations \eqref{EmBigO} and \eqref{Embound} are immediate consequences of \eqref{EmExplicit}. 
\end{proof}

\section{Proof of Theorems \ref{thm4} and \ref{thm5}}

It is well-known for the Mertens problem that the existence of a zero off the half-line or of a zero with multiplicity $m_\rho \ge 2$ creates large oscillations in the error term.
Thus if $\rho = \Theta+i \gamma $ is a zero of multiplicity $m_\rho\ge 1$, then 
\[ M(x) = \Omega_\pm \left( x^{\Theta}\log^{m_\rho -1}x\right), \]
see \cite[p. 467]{MontgomeryVaughan2007}.
The same method implies for $E_m(x)$ that for $m\ge 1$
\begin{equation} \label{Theta} E_m(x) = \Omega_\pm \left( x^{m + \frac{\Theta}{2} -1}\log^{m_\rho -1}x\right), \end{equation}
and therefore in particular if the Riemann Hypothesis is false or if there is a zero that is not simple then \eqref{thm5eq1} follows. Thus we are left to consider the case where the Riemann Hypothesis is true and all the zeros are simple, which we henceforth assume. 

Our proof is a simple application of a theorem of Ingham \cite{Ingham1942}. Ingham stated his result in terms of Laplace transforms by making the change of variables $x=e^u$ in the Mellin transform; the form we give below is for the Mellin transform in the form given in \cite{AndersonStark1980}, see also \cite[Theorem 11.12]{BatemanDiamond2004}. Note we are considering here Mellin transforms that have a sequence of simple poles going to infinity on the line $s=\sigma_0$ symmetrically above and below the real axis.

\begin{proposition}[Ingham] Let $A(x)$ be a bounded Riemann integrable function in any finite interval $1\le x\le X$, and suppose that
\begin{equation}\label{MellinA} \widehat{A}(s) := \int_1^\infty \frac{A(x)}{ x^{s+1}} \, dx \end{equation}
converge for $\sigma >\sigma_1$ and $\widehat{A}(s)$ has an analytic continuation to $\sigma > \sigma_0$. Further, suppose we have a sequence of real numbers $0<\check \gamma_1 < \check \gamma_2 <\check \gamma_3 < \ldots $, with $\check \gamma_n \to \infty$, and define $\check \gamma_0 =0$ and $\check \gamma_{-n} = -\check \gamma_n$. Also define a corresponding sequence of complex numbers $r_n\neq 0$, $r_{-n} = \overline{r}_n$, and $r_0$ real and allowed to be zero. By \eqref{MellinA} we obtain the Mellin transform pair, for any $T>0$, 
\begin{equation} S_T(x) = x^{\sigma_0}\sum_{|\check \gamma_n|\le T} r_n x^{i\check \gamma_n}, \qquad \widehat{S}_T(s) = \sum_{|\check \gamma_n|\le T}\frac{ r_n}{s-(\sigma_0+i\check \gamma_n)}. \end{equation}
If  $\widehat{A}(s)-\widehat{S}_T(s)$ can be analytically continued to the region $\sigma\ge \sigma_0$, $-T\le t\le T$, for some $T>0$,  then we have
\begin{equation} \liminf_{x\to \infty}\frac{A(x)}{x^{\sigma_0}}\le S_T^*(x_0) \le \limsup_{x\to \infty}\frac{A(x)}{x^{\sigma_0}}, \end{equation}
for any real number $x_0$, where
\begin{equation} S_T^*(x) = \sum_{|\check \gamma_n|\le T}\left(1- \frac{|\check\gamma_n|}{ T}\right) r_n x^{i\check \gamma_n} = r_0+ 2\, {\rm Re}\sum_{0<\check \gamma_n\le T}\left(1- \frac{\check\gamma_n}{ T}\right) r_n x^{i\check \gamma_n}. \end{equation}
\end{proposition}

Notice 
\[x^{\sigma_0}S_T^*(x) = \frac{1}{T}\int_0^T S_t(x)\, dt \]
has smoothed out some of the variation in $S_T(x)$, but in applications where we take $T\to \infty$ there is no loss in using $S_T^*(x)$ in place of $S_T(x)$. In numerical work however where $T$ is fixed, there is a better choice than $S_T^*(x)$, see \cite{OdlyzkoteRiele} and \cite{BestTrudgian2015}.

To apply this Proposition to $E_m(x)$, we note by integration by parts that for $\sigma >1$
\[ \int_1^\infty \frac{S_m(x)}{x^{s+m+1}}\, dx = \sum_{k=1}^\infty \mathfrak{S}(k) \int_k^\infty \frac{(x-k)^m}{x^{s+m+1}}\, dx = \frac{m!}{s(s+1)(s+2)\cdots (s+m)}F(s), \]
which may also be obtained from \eqref{Scontour} and \eqref{calF} with the Mellin inversion formula. On substituting \eqref{formula1} into this equation we obtain
\[ \int_1^\infty \frac{E_m(x)}{x^{s+m+1}}\, dx = \frac{m!}{s(s+1)(s+2)\cdots (s+m)}F(s) - \frac{\frac{1}{m+1}}{s-1} + \frac{\frac12}{s^2}-\frac{\frac12(H_m-\gamma-\log 2\pi)}{s}.\]
The case $m=1$ was used in \cite{GS2020}. In the Proposition, we take 
\[ A(x) = \frac{E_m(x)}{x^{m}}, \qquad \widehat{A}(s) = \frac{m!}{s(s+1)(s+2)\cdots (s+m)}F(s) - \frac{\frac{1}{m+1}}{s-1} + \frac{\frac12}{s^2}-\frac{\frac12(H_m-\gamma-\log 2\pi)}{s},\]
and see from Lemmas \ref{F(s)lem} and \ref{lemma2.2} that $\widehat{A}(s)$ is analytic for $\sigma > -\frac34$ and meromorphic for $\sigma > -1$ with simple poles at $s=\frac{\rho}{2} -1= -\frac34 + i\frac{\gamma}{2}$, where $\rho$ are the complex zeros of $\zeta(s)$, and the residues at these simple poles are $a(\rho)$. Thus in the Proposition we have $\sigma_0=-\frac34$, $\check \gamma_n = \frac{\gamma_n}{2}$, $r_0=0$, and $r_n = a(\rho_n)$. Therefore we conclude that for any number $x_0$
\begin{equation} \label{lasteq} \liminf_{x\to\infty}\frac{E_m(x)}{x^{m-\frac34}}\le 2\ {\rm Re}\ \sum_{0<\gamma_n\le 2T}\left(1-\frac{\gamma_n}{2T}\right) a(\rho_n)(x_0)^{i\frac{\gamma_n}{2}} \le \limsup_{x\to \infty}\frac{E_m(x)}{x^{m-\frac34}}. \end{equation}

\begin{proof}[Proof of Theorem 4] Since $\gamma_1 = 14.134725\ldots$ and 
$\gamma_2 =21.022039\ldots$, we choose $T=10$ in \eqref{lasteq}, and obtain
\[ \liminf_{x\to\infty}\frac{E_m(x)}{x^{m-\frac34}}\le 2 \left(1-\frac{\gamma_1}{20}\right){\rm Re}\left( a(\rho_1)(x_0)^{i\frac{\gamma_1}{2}}\right) \le \limsup_{x\to \infty}\frac{E_m(x)}{x^{m-\frac34}}.\]
We can clearly find values for $x_0$ so that
\[ a(\rho_1)(x_0)^{i\frac{\gamma_1}{2}} = |a(\rho_1)|e^{i( \arg(a(\rho_1)+\frac{\gamma_1}{2}\log(x_0))} = \pm|a(\rho_1)|,\]
and Theorem \ref{thm4} follows. 
\end{proof}

\begin{proof}[Proof of Theorem 5] We will prove the limsup parts of theorem since the liminf is handled the same way. Pick an $\epsilon>0$. If the imaginary parts of the zeros are linearly independent over the integers, by Kronecker's theorem we can find values of $x_0=x_0(T)$ so that 
\begin{equation} \label{nextlast}\limsup_{x\to \infty}\frac{E_m(x)}{x^{m-\frac34}}\ge 2\ {\rm Re} \sum_{0<\gamma_n\le 2T}\left(1-\frac{\gamma_n}{2T}\right) a(\rho_n)(x_0)^{i\frac{\gamma_n}{2}} > 2(1-\epsilon) \sum_{0<\gamma_n\le 2T}\left(1-\frac{\gamma_n}{2T}\right) |a(\rho_n)|.\end{equation}

\noindent{\it Case 1.} Suppose $\sum_\gamma |a(\rho)|$ diverges. Then \eqref{thm5eq1} holds since
\[ \limsup_{x\to \infty}\frac{E_m(x)}{x^{m-\frac34}} \gg \sum_{0<\gamma_n\le T} |a(\rho_n)| \to \infty, \qquad \text{as} \ \ T\to \infty.\]

\noindent{\it Case 2.} Suppose $c_m=\sum_\gamma |a(\rho)|$ converges. 
With the $\epsilon$ in \eqref{nextlast}, we have on taking $T$ sufficiently large
\[\begin{split} \limsup_{x\to \infty}\frac{E_m(x)}{x^{m-\frac34}} & > 2(1-\epsilon) \sum_{0<\gamma_n\le 2\epsilon T}\left(1-\frac{\gamma_n}{2T}\right)|a(\rho_n)| \\ & > 2(1-\epsilon)^2\sum_{0<\gamma_n\le 2\epsilon T}|a(\rho_n)|\\ &>(1-\epsilon)^3c_m .\end{split}\]
and we conclude $\limsup_{x\to \infty}\frac{E_m(x)}{x^{m-\frac34}}\ge c_m$ and \eqref{thm5eq2} follows.

It remains to prove that $\sum_\gamma|a(\rho)|$ diverges when $m=1$. 
By \eqref{FE} and \eqref{calG bounded} we have
\[\begin{split} |a(\rho)| &\asymp \frac{ |\zeta(\frac{\rho}{2}-1)|\zeta(\frac{\rho}{2})|}{|\rho^2 \zeta'(\rho)|}
\\& \asymp \frac{ |\zeta(2-\frac{\rho}{2})|\zeta(1-\frac{\rho}{2})|}{|\rho^{\frac12} \zeta'(\rho)|} \\&
\gg \frac{1}{|\rho^{\frac12 +\epsilon}\zeta'(\rho)|},
\end{split}\]
since assuming the Riemann Hypothesis $\left|\frac{1}{\zeta(s)}\right| =O( (|t|+3)^\epsilon)$ for $\sigma > \frac12$, see \cite[Eq. (14.2.6)]{Titchmarsh}.
By Theorem 15.6 of \cite{MontgomeryVaughan2007}, the Riemann Hypothesis implies that
\[ \sum_{0<\gamma\le T} \frac{1}{|\zeta'(\rho)|} \gg T, \]
and the divergence follows from
\[ \sum_{0< \gamma \le T}|a(\rho)| \gg \sum_{0< \gamma \le T} \frac{1}{|\gamma^{\frac12 +\epsilon}\zeta'(\rho)|}\gg \frac{1}{T^{\frac12+\epsilon}} \sum_{0<\gamma\le T} \frac{1}{|\zeta'(\rho)|} \gg T^{\frac12-\epsilon}.\]

\end{proof}

%%%%%%%%%%%%%%%%%%%%%%%%%%%%%%%%
%\bibliographystyle{amsplain}

\end{document}